\documentclass[a4paper, 10pt]{article}
\usepackage[latin1]{inputenc}
\usepackage{amsmath}
\usepackage{amssymb}
\usepackage{amsthm}
\usepackage{graphicx}
\usepackage[all]{xy}
\usepackage[english]{babel}

\title{On the representation ring of the polynomial algebra over a perfect field}
\author{Erik Darp\"o, Martin Herschend}

\newcommand{\edge}[1]{\ar@{-}[#1]}
\newcommand{\iso}{\;\tilde\rightarrow\;}

\newcommand{\spn}{\mathop{\rm span}\nolimits}

\newcommand{\cat}[1]{\mathop{\rm #1}\nolimits}
\newcommand{\modu}{\!-\!\mathop{\rm mod}}

\DeclareMathOperator{\cha}{char}

\newcommand{\latop}[2]{\genfrac{}{}{0pt}{}{#1}{#2}}
\newcommand{\G}[2]{\mathcal{G}(#1\slash #2)}
\newcommand{\eg}{e.g.}
\newcommand{\ie}{i.e.}
\renewcommand{\ge}{\geqslant} \renewcommand{\le}{\leqslant}
\def\<{\langle} \def\>{\rangle}
\def\a{\alpha} \def\l{\lambda}
\def\R{\mathbb{R}}
\def\C{\mathbb{C}}

\newtheorem{thm}{Theorem}
\newtheorem{cor}{Corollary}
\newtheorem{pro}{Proposition}
\newtheorem{lem}{Lemma}

\begin{document}
\date{}
\maketitle

\begin{abstract}
\noindent We consider the tensor product of modules over the polynomial algebra corresponding to the usual tensor product of linear operators. We present a general description of the representation ring in case the ground field $k$ is perfect. It is made explicit in the special cases when $k$ is real closed respectively algebraically closed. Furthermore, we discuss the generalisation of this problem to representations of quivers. In particular the representation ring of quivers of extended Dynkin type $\tilde{\mathbb{A}}$ is provided.
\end{abstract}

\noindent 
Keywords: representation ring, polynomial algebra, quiver representation, tensor product, Clebsch-Gordan problem.

\section{Introduction}
Let $k$ be a field. We consider the polynomial algebra $k[x]$ as a bialgebra with comultiplication $k[x] \rightarrow k[x]\otimes k[x]$, defined by $x \mapsto x\otimes x$. This defines a tensor product on the category of finite-dimensional left modules of $k[x]$. The objects in this category correspond to linear operators on finite-dimensional vector spaces and their tensor product is given by the usual tensor product of linear operators.

By the structure theorem for finitely generated modules over principal ideal domains
we have the following result. 

\begin{thm}\label{class}
The modules $k[x]/f(x)^s$, where $s$ is a positive integer and $f(x)\in k[x]$ is
irreducible and monic, classify all indecomposable finite-dimensional $k[x]$-modules up to isomorphism.
\end{thm}

By Theorem \ref{class}, it suffices to decompose $k[x]/f(x)^s \otimes k[x]/g(x)^t$ for all positive integers $s,t$ and irreducible monic polynomials $f,g$, to determine $V\otimes W$ for all finite-dimensional $k[x]$-modules $V$ and $W$, \ie\ solve the Clebsch-Gordan problem for $k[x]$.

If $k$ is algebraically closed this amounts to finding the Jordan normal form for the
Kronecker product of any two Jordan blocks. In case $k$ has characteristic zero this
problem has been solved by many authors. The first to our knowledge is Aitken
\cite{aitken35}. However, this reference seems relatively unknown as there are several
independent solutions \eg\ by Huppert \cite{huppert90}, and by Martsinkovsky and Vlassov
\cite{MaVl}.

As is noted in \cite{iwamatsu07a}, this problem appears naturally in the construction of graded Frobenius algebras by Wakamatsu in \cite{wakamatsu03}.
Denote by $J_{\l}(l)$ the Jordan block of size $l \in \mathbb{N}\setminus\{0\}$ and
eigenvalue $\l \in k$. For any two matrices $A$ and $B$ we write $A \sim B$ if $A = TBT^{-1}$ for some invertible matrix $T$.

\begin{thm}\label{jordan}
For all $\l,\mu \in k\setminus\{0\}$ and positive integers $l,m$ the following fomulae hold:
\begin{enumerate}
\item $J_{\l}(l) \otimes J_{\mu}(m) \sim \bigoplus_{i=0}^{l-1} J_{\l\mu}(l+m-2i-1)$ if $l\le m$
and $\cha k=0$,
\item $J_{\l}(l) \otimes J_{0}(m) \sim lJ_{0}(m)$,
\item $J_{0}(l) \otimes J_{0}(m) \sim (m-l+1)J_{0}(l) \oplus \bigoplus_{i=1}^{l-1}2J_{0}(i)$ if $l\le m$.
\end{enumerate}
\end{thm}
In positive characteristic we know of no explicit formula for the decomposition of $J_{\l}(l) \otimes J_{\mu}(m)$ when $\l,\mu \in k\setminus\{0\}$. However, Iima and Iwamatsu in \cite{iwamatsu07a} give an algorithm for computing the decomposition in this case. Also, we will later present an alternative method for achieving this.

The solution to the Clebsch-Gordan problem is encoded in the representation ring $R$ of $k[x]$. As an abelian group $R$ is freely generated by the isoclasses of indecomposable finite-dimensional $k[x]$-modules. For two such isoclasses $[V], [W]$ their product is
\[
[V][W] = \sum_{i\in I} [U_i],
\]
where 
\[
V\otimes W \iso \bigoplus_{i\in I} U_i,
\]
is the decomposition of $V\otimes W$ into indecomposables. The representation ring is commutative with identity element $[k[x]/(x-1)]$. Define the ring morphism
\[
\dim : R \rightarrow \mathbb{Z},
\]
by $\dim ([V]) = \dim V$.

In the present article we give a general description of $R$ for any perfect ground field $k$. Under various restrictions on the ground field this description is then made more explicit.

\section{General description of the representation ring}

From here on, let $K$ be the algebraic closure of $k$.
To each $k[x]$-module $V$ we associate the $K[x]$-module $K \otimes V$. This allows us to
reduce many problems to the algebraically closed case by using the following lemma due to
Noether, see \cite{deuring33}. We include the proof for completeness. 

\begin{lem}\label{extension}
Let $F$ be an algebraic field extension of $k$ and $A$ an associative $k$-algebra with identity. Further let $V$ and $W$ be finite-dimensional $A$-modules. If $F\otimes V$ and $F\otimes W$ are isomorphic as $F\otimes A$-modules, then $V$ and $W$ are isomorphic as $A$-modules.
\end{lem}
\begin{proof}
First assume that $F$ is finite-dimensional over $k$ and let $b_1, \ldots b_n$ be a $k$-basis of $F$. As $A$-modules, $F\otimes V = \bigoplus_{i=1}^n b_i\otimes V \iso nV$ and $F \otimes W \iso nW$. Now assume that $F\otimes V$ and $F\otimes W$ are isomorphic as $F\otimes A$-modules. Then they are also isomorphic as $A$-modules and $nV \iso nW$. By the Krull-Schmidt Theorem, $V \iso W$.

Now assume that $F$ is algebraic over $k$, but not necessarily finite. Let $\phi : F\otimes V \iso F\otimes W$ be a $F\otimes A$-module isomorphism and choose $k$-bases in $V$ and $W$. These are also $F$-bases of $F \otimes V$ and $F\otimes W$ respectively. Let $T$ be the matrix corresponding to $\phi$ in the chosen bases. Since $F$ is algebraic over $k$ there is a finite extension $E$ of $k$ that contains the elements of $T$. Hence $T$ defines an $E\otimes A$-module morphism $E\otimes V \iso F\otimes W$. By our result in the previous paragraph, $V \iso W$. 
\end{proof}

\begin{pro}\label{zero}
Let $s$ and $t$ be positive integers and $f(x)\in k[x]$ irreducible with $f(0) \not = 0$. Then the following fomulae hold.
\begin{enumerate}
\item $k[x]/x^s \otimes k[x]/f(x)^t \iso t(\deg f) k[x]/x^s$.
\item $k[x]/x^s \otimes k[x]/x^t \iso (t-s+1)k[x]/x^s \oplus \bigoplus_{i=1}^{s-1}2k[x]/x^i$ if $s \le t$.
\end{enumerate}
\end{pro}
\begin{proof}
Consider $f(x) \in K[x]$. We may assume that $f$ is monic and thus $f(x)= \prod_{\l\in\Lambda}
(x-\l)^{d_\l}$ for some finite subset $\Lambda\subset K\setminus\{0\}$, $d_\l\in \mathbb{N} \setminus \{ 0\}$ and $\deg f = \sum_{\l\in\Lambda} d_\l$. By the Chinese Remainder Theorem,
\[
K[x]/f(x)^t \iso \bigoplus_{\l\in\Lambda} K[x]/(x-\l)^{td_\l}.
\]
Now
\[\begin{split}
K[x]/x^s \otimes K[x]/f(x)^t 
&\iso K[x]/x^s \otimes \bigoplus_{\l\in\Lambda} K[x]/(x-\l)^{td_\l} \\ 
&\iso \bigoplus_{\l\in\Lambda} td_\l K[x]/x^s = t(\deg f) K[x]/x^s,
\end{split}\]
by Theorem \ref{jordan}. We also obtain
\[
K[x]/x^s \otimes K[x]/x^t \iso (t-s+1)K[x]/x^s \oplus\bigoplus_{i=1}^{s-1}2K[x]/x^i
\]
if $s \le t$ from Theorem \ref{jordan}. By Lemma \ref{extension}, these fomulae also hold when we replace $K$ by $k$.
\end{proof}

From Proposition \ref{zero} follows that the elements $[k[x]/x^s]$ span an ideal $I$ in $R$. Moreover, Proposition \ref{zero} describes how the elements in $R$ act on $I$. Note that if a $k[x]$-module $V$ contains no direct summand isomorphic to $k[x]/x^s$, then $[V]w = (\dim V) w$ for all $w \in I$. 

From now on we assume that $k$ is perfect, i.e. all irreducible polynomials over $k$ have
distinct zeros. Let $\l \in K\setminus\{0\}$ and $l$ be a positive integer. The matrix $J_{1}(l)\otimes J_{\l}(1)$ is conjugate to $\l J_{1}(l)$, which in turn is conjugate to $J_\l(l)$ via a rescaling of the basis vectors. Hence $J_{\l}(l) \otimes J_{\mu}(m) \sim J_{1}(l) \otimes J_{1}(m) \otimes J_{\l}(1) \otimes J_{\mu}(1) \sim J_{1}(l) \otimes J_{1}(m) \otimes J_{\l\mu}(1) $ for all $\l,\mu \in k\setminus\{0\}$.  To find the Jordan decomposition of $J_{\l}(l) \otimes J_{\mu}(m)$ it is therefore enough to decompose $J_{1}(l) \otimes J_{1}(m)$. The following lemma extends this result to arbitrary perfect fields.

\begin{lem}\label{separate}
For any positive integer $s$ and irreducible polynomial $f(x)\in k[x]$ with $f(0) \not =
0$, the $k[x]$-modules $k[x]/f(x)^s$ and $k[x]/(x-1)^s \otimes k[x]/f(x)$ are isomorphic.
\end{lem}
\begin{proof}
Let $\l \in K \setminus\{0\}$. Since $J_{\l}(s)$ is conjugate to $J_{1}(s)\otimes J_{\l}(1)$ we have $K[x]/(x-\l)^s \iso K[x]/(x-1)^s \otimes K[x]/(x-\l)$.

Now let $f(x)\in k[x]$ be irreducible and monic with $f(0) \not = 0$. Since $k$ is perfect
$f(x)= \prod_{\l\in\Lambda} (x-\l)$ for some finite $\Lambda \subset K\setminus\{0\}$ and 
\[\begin{split}
K\otimes (k[x]/f(x)^s)
&\iso K[x]/f(x)^s \\
&\iso \bigoplus_{\l\in\Lambda} K[x]/(x-\l)^{s}\\
&\iso \bigoplus_{\l\in\Lambda} K[x]/(x-1)^s \otimes K[x]/(x-\l) \\
&\iso K[x]/(x-1)^s \otimes \bigoplus_{\l\in\Lambda} K[x]/(x-\l) \\
&\iso K[x]/(x-1)^s \otimes K[x]/f(x) \\
&\iso K\otimes(k[x]/(x-1)^s \otimes k[x]/f(x)).
\end{split}\] 
The result now follows from Lemma \ref{extension}.
\end{proof}

Let $R'$ be the $\mathbb{Z}$-span in $R$ of all $[k[x]/(x-1)^s]$, and $\bar R$ be the $\mathbb{Z}$-span in $R$ of all $[k[x]/f(x)]$, where $f(x)\in k[x]$ is irreducible and $f(0) \not = 0$.

\begin{pro} \label{subring}
\begin{enumerate}
\item \label{sub1}
  The sets $R'$ and $\bar R$ are subrings of $R$. 
\item \label{sub2}
  Let $f(x),g(x)\in k[x]$ be irreducible with $f(0)\neq0\neq g(0)$, and $\Lambda$ and $M$
  their respective sets of zeros in $K$.
  If $$k[x]/f(x)\otimes k[x]/g(x)\iso \bigoplus_{j\in J} k[x]/h_j(x)$$ then the zeros in
  $K$ of all $h_j(x)$, counting repetitions, are precisely the numbers $\l\mu$, with
  $(\l,\mu)\in \Lambda\times M$.
 \end{enumerate}
\end{pro}
\begin{proof}
First note that $1_R = [k[x]/(x-1)] \in R'\cap \bar R$. We proceed to show that $R'$ and $\bar R$ are closed under multiplication. 

The matrix $J_1(s)\otimes J_1(t)$ is conjugate to $\mathbb{I}_{st}+N$ for some nilpotent matrix $N$. Since every nilpotent matrix is conjugate to a direct sum of Jordan blocks with eigenvalue zero, we get
\[
k[x]/(x-1)^s \otimes k[x]/(x-1)^t \iso \bigoplus_{i = 1}^n k[x]/(x-1)^{m_i}
\]
for some integers $m_i$, and thus $R'$ is a subring of $R$.

Let $f(x),g(x) \in k[x]$ be irreducible and monic with $f(0) \not = 0 \not = g(0)$. Decompose $f(x) = \prod_{\l \in \Lambda}(x-\l)$ and $g(x) = \prod_{\mu \in M}(x-\mu)$ in $K[x]$. Further, decompose 
\[
k[x]/f(x) \otimes k[x]/g(x) \iso \bigoplus_{j\in J} k[x]/h_j(x)^{d_j}
\]
for some irreducible polynomials $h_j(x) \in k[x]$. Then
\begin{equation*}
\begin{split}
\bigoplus_{j\in J} K[x]/h_j(x)^{d_j}
&\iso K[x]/f(x) \otimes K[x]/g(x)  \\
&\iso \bigoplus_{\latop{\l \in \Lambda}{\mu\in M}} K[x]/(x-\l)\otimes K[x]/(x-\mu) \\
&\iso \bigoplus_{\latop{\l \in \Lambda}{\mu\in M}} K[x]/(x-\l\mu).
\end{split}
\end{equation*}
By the Krull-Schmidt Theorem, the only possibility is that $d_j = 1$ and $h_j(0) \not = 0$ for all $j \in J$. Hence $\bar R$ is a subring of $R$.
Moreover, the zeros of the polynomials $h_j(x)$ are the products $\l\mu$ of zeros of
$f(x)$ and $g(x)$, as asserted in the second part of the proposition.
\end{proof}

Define a ring structure on $R' \otimes_\mathbb{Z} \bar R \oplus I$ by $(a\otimes b)w = \dim (a) \dim (b) w$ for all $a \in R'$, $b \in R'$ and $w \in I$.

\begin{thm}\label{repring}
The $\mathbb{Z}$-linear map
\[
\phi : R' \otimes_\mathbb{Z} \bar R \oplus I \rightarrow R,
\]
defined by $\phi(a\otimes b + w) = ab + w$ is a ring isomorphism.
\end{thm}
\begin{proof}
By Lemma \ref{separate}, $\phi(R' \otimes_\mathbb{Z} \bar R)$ is spanned by all $[k[x]/f(x)^s]$, where $f(x)\in k[x]$ is irreducible, $f(0) \not = 0$ and $s \in \mathbb{N} \setminus \{0\}$. Hence $R  = \phi(R' \otimes_\mathbb{Z} \bar R) \oplus I$. Moreover $\phi$ induces a bijection between $\mathbb{Z}$-bases, and is therefore a bijection. 

To show that $\phi$ is a ring morphism it is enough to check that $\phi((a\otimes b)w) = \phi(a \otimes b)\phi(w) $ for all $a \in R'$, $b \in \bar R$ and $w \in I$. By Proposition \ref{zero},
\[
\phi(a \otimes b)w = abw = \dim(a)\dim(b)w = \phi(\dim(a)\dim(b)w) = \phi((a\otimes b)w).
\]
Hence $\phi$ is an isomorphism of rings.
\end{proof}

Since the structure of the ideal $I$ is completely described by Proposition \ref{zero} and in fact does not depend on the ground field $k$, Theorem \ref{repring} reduces the problem of describing $R$ to describing each of the rings $R'$ and $\bar R$. 

\section{Explicit description of the representation ring}
In this section, we investigate the structure of the rings $R'$ and $\bar{R}$. It turns out that the ring $R'$ only depends on the characteristic of $k$ and when $k$ is algebraically or real closed the ring $\bar{R}$ has a fairly simple description. We also present the explicit Clebsch-Gordan formulae for the decomposition of the tensor product of arbitrary finite-dimensional $k[x]$-modules, in case $k$ is real closed.

\subsection{The ring $R'$}
Denote $k[x]/(x-1)^s$ by $V_s$ for all $s \ge 0$, and set $v_s = [V_s]$. 
In particular, $v_0 = 0$ and $v_1 = 1$ in $R$. Recall that $R'$ is a subring of $R$ and freely generated, as an abelian group, by the set $\mathcal{V} = \{v_s \;|\; s \ge 1\}$.

We start with a description of $R'$ in characteristic zero, derived form Theorem \ref{jordan}. It is included here to contrast the case of positive characteristic.

\begin{thm}\label{charzero}
Assume that the characteristic of $k$ is zero. The ring morphism
\[
\phi : \mathbb{Z}[T] \rightarrow R',
\]
defined by $T \mapsto v_2$ is an isomorphism.
\end{thm}
\begin{proof}
Let $Z_s$ be the span of the elements $v_t$, where $1 \le t \le s$. By Theorem \ref{jordan},
\[
v_2v_s = v_{s-1} + v_{s+1}
\]
for all positive integers $s$. In particular $v_2 Z_s \subset Z_{s+1}$.

We show by induction that $v_2^s \in v_{s+1} + Z_s$ for all $s \ge 0$, from which the theorem follows. First note that $v_2^0 = 1 = v_1 \in v_1 + Z_0$. Now assume that $v_2^s \in v_{s+1} + Z_s$ for some natural number $s$. Then $v_2^{s+1} \in  v_2(v_{s+1} + Z_s) \subset v_{s+2} +v_s + Z_{s+1} = v_{s+2} + Z_{s+1}$.
\end{proof}

Note that $f_{s}(T) = \phi ^{-1}(v_s)$ satisfies the recurrence relation
\[
f_{s+1} = f_2f_s -  f_{s-1},
\]
and $f_1(T) = 1$, $f_2(T) = T$. Set $U_s = f_{s+1}(2T)$. Then $U_0(T) = 1$, $U_1(T) = 2T$, and
\[
U_{s+1}(T) = f_{s+2}(2T) = f_2(2T)f_{s+1}(2T) - f_{s}(2T) = 2TU_{s}(T) - U_{s-1}(T).
\]
Hence the polynomials $f_{s+1}(2T) = U_s(T)$ are the Chebyshev polynomials of the second kind. 

By Theorem \ref{charzero}, the ring $R'$ is generated by the single element $v_2$ in case the ground field $k$ has characteristic zero. As we shall see, this is very far from the behaviour of $R'$ in positive characteristic. 

Assume that $k$ has characteristic $p>0$. We proceed to describe $R'$ by relating it to
the representation rings of cyclic $p$-groups, which are described by Green in \cite{green61}.  Let
$\alpha \in \mathbb{N}$ and $\sigma_\alpha$ be a chosen generator of $C_\alpha$, the cyclic group of order $q=p^\alpha$. Denote the representation ring of $k C_\alpha$ by $A_\alpha$. The indecomposable $kC_\alpha$-modules are classified by the modules
\[
kC_\alpha /(\sigma_\alpha-1)^s
\]
where $1 \le s \le q$. Let $\beta \ge \alpha$. The homomorphism $C_\beta \rightarrow C_\alpha$ defined by $\sigma_\beta \mapsto \sigma_\alpha$ allows us to view any $kC_\alpha$-module as a $kC_\beta$-module. Hence we may regard $A_\alpha$ as  subring of $A_\beta$, by identifying $[kC_\alpha /(\sigma_\alpha-1)^s]$ and $[kC_\beta /(\sigma_\beta-1)^s]$ for all $1 \le s \le q$. 

Accordingly, for every $s \ge 1$ we denote each of the elements $[kC_\alpha /(\sigma_\alpha-1)^s]$, where $p^\alpha \ge s$, by $u_s$.  The elements $u_s$ span the ring
\[
A = \bigcup_{\alpha \in \mathbb{N}} A_\alpha
\]
freely over $\mathbb{Z}$.

\begin{pro}\label{charp}
Assume that the characteristic of $k$ is $p>0$. The $\mathbb{Z}$-linear map
\[
\phi : A \rightarrow R',
\]
defined by $u_s \mapsto v_s$ is a ring isomorphism.
\end{pro}
\begin{proof}
As abelian groups $A$ and $R'$ are freely generated by the elements $u_s$ and $v_s$ respectively. It therefore suffices to check that $\phi$ respects multiplication.

Let $\alpha \in \mathbb{N}$, $q=p^\alpha$ and $1 \le s \le q$. The $kC_\alpha$-module $kC_\alpha /(\sigma_\alpha-1)^s$ has the basis $\{(\sigma_\alpha-1)^i \;|\; 0 \le i \le s-1\}$. Moreover,
\[
\sigma_\alpha (\sigma_\alpha-1)^i = (\sigma_\alpha-1)^{i+1} + (\sigma_\alpha-1)^i.
\]
Hence $\sigma_\alpha$ has the matrix $J_1(s)$ in the aforementioned basis. Thus, if $J_1(s) \otimes J_1(t)$ has the Jordan decomposition
\[
\bigoplus_{i=1}^n J_1(m_i),
\]
then $u_su_t = \sum_{i=1}^n u_{m_i}$. On the other hand we also have $v_sv_t = \sum_{i=1}^n v_{m_i}$. Hence $\phi$ is an isomorphism of rings.
\end{proof}

From Proposition \ref{charp} we deduce that $R'$ is not generated by any finite subset. For all $\alpha \in \mathbb{N}$ set $R'_\alpha = \phi(A_\alpha)$ and $w_{\alpha}  = v_{p^\alpha+1}-v_{p^\alpha-1}$. We immediately obtain the following translation of \cite[Theorem~3]{green61}:

\begin{thm}\label{charprel}
Let $k$ be a field of characteristic of $p>0$ and $\alpha \in \mathbb{N}$. Set $q = p^\alpha$. Then
\[
w_{\alpha}v_r = \left\{
\begin{array}{lcr}
v_{r+q} - v_{q-r} & \mbox{if} & 1 \le r \le q \\
v_{r+q} + v_{r-q} & \mbox{if} &q < r \le (p-1)q \\
v_{r-q} +2v_{pq}- v_{(2p-1)q-r} & \mbox{if} & (p-1)q < r \le pq \\
\end{array} \right.
\]
Moreover this equation defines the multiplicative structure of $R'$.
\end{thm}

It follows that $R'$ is generated by $\mathcal{W} = \{w_\alpha \;|\; \alpha \in \mathbb{N}\}$. In fact, $R'_{\alpha+1} = R'_{\alpha}[w_\alpha]$. Consider the monomorphism $C_\alpha \rightarrow C_{\alpha +1}$, $\sigma_\alpha \mapsto \sigma_{\alpha+1}^p$. From it we obtain the restriction functor
\[
\cat{Res} : kC_{\alpha+1} \modu \rightarrow kC_{\alpha} \modu
\]
and its left adjoint, the induction functor
\[
\cat{Ind} : kC_{\alpha} \modu \rightarrow kC_{\alpha +1} \modu,
\]
where $kC_{\alpha} \modu $ denotes the category of all finite-dimensional $kC_{\alpha}$-modules. Let $1 \le s \le p^{\alpha+1}$ and write $s = tp +r$ for some $0 \le r < p$. It is straightforward to show that 
\[
\cat{Res}(kC_{\alpha+1} /(\sigma_{\alpha+1}-1)^s) \iso r kC_{\alpha} /(\sigma_{\alpha}-1)^{t+1} \oplus (p-r) kC_{\alpha} /(\sigma_{\alpha}-1)^{t}.
\] 
Also, for $1 \le s \le p^{\alpha}$ it holds that $\cat{Ind} (kC_{\alpha} /(\sigma_{\alpha}-1)^s) \iso kC_{\alpha+1} /(\sigma_{\alpha+1}-1)^{ps}$. Furthermore, recall that for all $kC_{\alpha+1}$-modules $V$ and $kC_{\alpha}$-modules $W$ the following formula holds:
\[
V \otimes \cat{Ind} W \iso \cat{Ind}(\cat{Res} V \otimes W).
\]

By Proposition \ref{charp}, The functors $\cat{Ind}$ and $\cat{Res}$ induce $\mathbb{Z}$-linear maps $\iota : R'  \rightarrow R'$ and $\rho : R'  \rightarrow R'$ respectively. From the observations above we obtain the following lemma:

\begin{lem}\label{indred}
Let $s \in \mathbb{N}$ and write $s = tp +r$ for some $0 \le r < p$. Further, let $v,w \in R'$. Then the following formulae hold
\begin{enumerate}
\item $\rho (v_s) = rv_{t+1} + (p-r)v_t$
\item $\iota (v_s) = v_{ps}$
\item $\iota(v)w = \iota(v \rho (w))$
\end{enumerate}
\end{lem}

From Lemma \ref{indred} follows that for each $\alpha \in \mathbb{N}$ the image $\mathcal{V}_\alpha$ of $\iota^{\alpha}$ is spanned by all elements of the form $v_{p^\alpha s}$. Moreover, $\mathcal{V}_\alpha \subset R'$ is an ideal. 

\begin{pro}
For all $\alpha \in \mathbb{N}$
\[
\mathcal{V}_\alpha = (v_{p^\alpha}).
\]
\end{pro}
\begin{proof}
Since $v_{p^\alpha} = \iota^\alpha(v_1)$, we have $(v_{p^\alpha}) \subset \mathcal{V}_{\alpha}$ for each $\alpha \in \mathbb{N}$. On the other hand note that 
\[
\rho (w_{\alpha+1}) = \rho(v_{p^{\alpha+1}+1} - v_{p^{\alpha+1}-1}) = v_{p^\alpha+1} + (p-1)v_{p^\alpha} - (p-1)v_{p^\alpha} - v_{p^\alpha-1} = w_{\alpha}.
\]
Hence $w_{\alpha +1}\iota(v) = \iota(w_\alpha v)$ and thus for any sequence of natural numbers $(\alpha_i)_{i=1}^n$,
\[
\prod_{i=1}^n w_{\alpha_i+\alpha}v_{p^\alpha} = \prod_{i=1}^n w_{\alpha_i+\alpha}\iota^\alpha(1) = \prod_{i=1}^{n-1} w_{\alpha_i+\alpha}\iota^\alpha(w_{\alpha_n}).
\]
By induction, the right hand side equals $\iota^\alpha\left(\prod_{i=1}^n w_{\alpha_i}\right)$ and thus $\mathcal{V}_\alpha \subset (v_{p^\alpha})$.
\end{proof}

Even though we have quite a lot of information about the ring $R'$, the problem of decomposing $J_1(s) \otimes J_1(t)$, or equivalently writing $v_sv_t$ as a linear combination of the elements of $\mathcal{V}$, remains. Below we shall describe a method of doing this by translating between polynomials in the elements of $\mathcal{W}$ and linear combinations of the elements of $\mathcal{V}$.

To write the elements of $\mathcal{V}$ as polynomials in the elements of $\mathcal{W}$ we proceed as follows. Let $s > 1$ and write $s = q +r$ for some $\alpha \in \mathbb{N}$ and $1 \le r \le (p-1)q$, where again $q=p^{\alpha}$. By Theorem \ref{charprel},
\begin{equation}\label{recursive}
v_s = \left\{
\begin{array}{lcr}
w_\alpha v_r + v_{s-2r} & \mbox{if} & 1 \le r \le q \\
w_\alpha v_r - v_{s-2q} & \mbox{if} &q < r \le (p-1)q \\
\end{array} \right.
\end{equation} 
Since $r$, $s-2r$ and $s-2q$ are all strictly smaller than $s$, repeated use of Equation (\ref{recursive}) will eventually yield $v_s$ written as a polynomial in the elements of $\mathcal{W}$, \eg\ for $p=3$,
\[
v_8 = w_1v_5-v_2 = w_1(w_1v_2+1)-w_0 = w_1^2w_0+w_1-w_0.
\]

Theorem \ref{charprel} describes how to write $w_\alpha v_r$ as a linear combination the elements $v_s$, where $s \le p^{\alpha+1}$, for all $r \le p^{\alpha+1}$. This yields a method for writing polynomials in the elements of $\mathcal{W}$ as linear combination of the elements of $\mathcal{V}$. Indeed, let $(\alpha_i)_{i=1}^n$ be a sequence of natural numbers such that $\alpha_{i} \le \alpha_{i+1}$. Then $w_{\alpha_1} = v_{p^{\alpha_i} +1} - v_{p^{\alpha_i} -1}$ and we can write $w_{\alpha_2}w_{\alpha_1}$ as a linear combination of the elements of $\mathcal{V}$ using Theorem \ref{charprel}. Moreover, since the $v_s$ that appear all have $s \le p^{\alpha_2 +1} \le p^{\alpha_3 +1}$ we can do the same for $w_{\alpha_3}w_{\alpha_2}w_{\alpha_1}$. Continuing in this fashion we will eventually obtain $w_{\alpha_n} \cdots \; w_{\alpha_1}$ written as a linear combination of the elements of $\mathcal{V}$. To illustrate: if $p = 3$, then
\[\begin{split}
w_1^2w_0=w_1^2(v_2) = w_1(v_5-v_1) = v_8+v_2-v_4+v_2 = v_8 -v_4+2v_2.
\end{split}\]

To find the explicit decomposition of $v_sv_t$ one can write both $v_s$ and $v_t$ as polynomials in the elements of $\mathcal{W}$ using Equation (\ref{recursive}), multiply the polynomials obtained and then use the method described above to write the result as a linear combination of the elements of $\mathcal{V}$.  Alternatively one can use the algorithm presented in \cite{iwamatsu07a}, which in most cases probably is quicker.

\subsection{The ring $\bar R$} \label{Rstreck}

As an abelian group, $\bar{R}$ is freely generated by its elements $[k[x]\slash f(x)]$,
$f(x)$ being a monic, irreducible polynomial with $f(0)\neq 0$. 
Let $P=P_k$ be the subset of $k(x)$ formed by all quotients $f(x)/g(x)$, where $f(x)$ and $g(x)$
are  monic polynomials with non-zero constant term. Clearly, $P$ is an 
abelian group under multiplication, isomorphic to $(\bar{R},+)$ via the map
\begin{equation} \label{piso}
  P\to \bar R,\quad \frac{\prod_i f_i(x)}{\prod_j g_j(x)} \mapsto 
  \sum_i [k[x]\slash f_i(x)] - \sum_j [k[x]\slash g_j(x)],
\end{equation}
where $f_i(x),g_j(x)\in k[x]$ are irreducible.

We now define a ring structure on $P$ as follows: Given monic, irreducible polynomials
$f(x),g(x)\in k[x]$ with $f(0)\neq0\neq g(0)$, define
\begin{equation*}
  f(x)\star g(x) =  \prod_{\latop{\l\in\Lambda}{\mu\in M}} (x-\l\mu),
\end{equation*}
where $\Lambda$ and $M$ are the sets of zeros in $K$ of $f(x)$ and $g(x)$ respectively. 
Since the polynomials of this type freely generate $(P,\cdot)$, this definition extends to
the entire $P$ by linearity. Statement \ref{sub2} of Proposition~\ref{subring} implies,
that this is precisely what is needed to match the multiplicative structure in $\bar{R}$. Thus we get
the following proposition.

\begin{pro}\label{Pring}
  The operation $\star$ defines a ring structure on $P$. With this structure, the map given by
  (\ref{piso}) is an isomorphism of rings. 
\end{pro}
Note that $\deg(f\star g)=\deg f \deg g$.

Shifting for a moment focus to $K$, the algebraic closure of $k$, we observe that the ring
$P_K$ is freely generated, as abelian group, by $\{x-\l \; | \; \l \in K^\iota\}$, where 
$K^\iota = K\setminus \{0\}$ is the group of invertible elements in $K$. From the
definition of the multiplication in $P_K$ we find that $P_K$ is isomorphic to
$\mathbb{Z}K^\iota$, the group ring of $K^\iota$, via $(x-\lambda)\mapsto \l$.

Let $G = \G{K}{k}$ be the absolute Galois group of $k$. Then  $G$ acts on $P_K$, by
transforming the coefficients of elements in $P_K$. Hence we obtain a $G$-action on
$\mathbb{Z}K^\iota$, which in fact comes from the natural $G$-action on
$K^\iota$. Proposition~\ref{Pring} yields the following corollary:

\begin{cor} There is an isomorphism of rings:
\[
\bar{R} \iso ({\mathbb{Z}K^\iota})^G
\]
Where $({\mathbb{Z}K^\iota})^G$ denotes the ring of invariants under $G$.
\end{cor}
\begin{proof}
The ring $P_k$ is a subring of $P_K$ since $k(x) \subset K(x)$. Assume that $q(x) = f(x)/g(x) \in K(x)$, where $f(x)$ and $g(x)$ are  monic polynomials with non-zero constant term. Then $q(x)$ is fixed by $G$ if and only if $q(x) \in K^G(x) = k(x)$. Hence $P_k = P_K^G$. Since $P_K^G \iso ({\mathbb{Z}K^\iota})^G$ the result follows from Proposition~\ref{Pring}.
\end{proof}

Assume $f(x)$ and $g(x)$ are irreducible polynomials in $P$. 
In studying the product $f(x)\star g(x)$, we may consider their
zeros in splitting fields, instead of in the algebraic closure $K$. 
Thus let $E$ be a splitting field of $f(x)g(x)$, and $B$ and $C$ the splitting
fields inside $E$ of $f(x)$ and $g(x)$ respectively.
Denote by $\Lambda\subset B$ the set of zeros of $f(x)$ and by $M\subset C$ the set of zeros of
$g(x)$.
Set $F=B\cap C$. Given any element $\a\in E$ and intermediate field $L$ of $E\supset k$,
denote by $m_L(\a)(x)\in L[x]$ the minimal polynomial of $\a$ over $L$.
Note that, since $k$ is perfect, $E\supset L$ is Galois, and 
\begin{equation} \label{rotarunt}
m_L(\a)(x)=\prod_{\beta\in G\cdot\a}(x-\beta),
\end{equation}
where $G=\G{E}{L}$ is the Galois group of $E\supset L$ and $G\cdot \a$ the orbit of $\a$
under $G$. 

We first consider the case when $B$ and $C$ are linearly disjoint, that is, when $F=k$.
\begin{pro} \label{linob}
  Suppose $F=k$, and let $\l\in B$ and $\mu\in C$ be zeros of $f(x)$ and $g(x)$
  respectively. 
  Now $f(x)\star g(x)=m_k(\l\mu)(x)^l$, where $l=\frac{\deg f \deg g}{\deg m_k(\l\mu)}$ is
  a common divisor of $\deg f$ and $\deg g$.
  Moreover, $k$ contains $l$ distinct $l$th roots of unity.
  In particular, $l$ cannot be a multiple of $\cha k$.
\end{pro}

\begin{proof}
Let $G=\G{E}{k}$ be the Galois group of the extension $E\supset k$. The Fundamental
Theorem of Galois theory implies (see \eg\ \cite[Corollary~91]{rotman}) that the map
\begin{equation*}
G\to\G{B}{k}\times\G{C}{k},\:\sigma\mapsto(\sigma|_B,\sigma|_C)
\end{equation*}
is an isomorphism of groups. Hence 
\begin{equation*}
  G\cdot\l\mu=\{\sigma(\l)\tau(\mu) \mid \sigma\in\G{B}{k},\: \tau\in\G{C}{k}\}=
  \{\a\beta \mid \a\in\Lambda,\:\beta\in M \}.
\end{equation*}
Thus the set of zeros of $f(x)\star g(x)$ equals the set of zeros of $m_k(\l\mu)(x)$.
As the latter polynomial is irreducible, this means that $f(x)\star g(x)=m_k(\l\mu)(x)^l$
for some positive integer $l$. 

Now there exist precisely $l$ distinct elements $\l_1,\ldots,\l_l\in \Lambda$ with
corresponding $\mu_1,\ldots,\mu_l\in M$ (say $\l_1=\l,\:\mu_1=\mu$) such that
$\l_i\mu_i=\l\mu$.
Set $a_i=\frac{\l_i}{\l}=\frac{\mu}{\mu_i}\in B\cap C=k$.

First note that multiplication with any $a_i$ permutes $\Lambda$. Because for any
$\a\in\Lambda$, there exists a $\sigma\in\G{B}{k}$ such that $\a=\sigma(\l)$, and hence
$a_i\a=a_i\sigma(\l)=\sigma(\a_i\l)=\sigma(\l_i)\in\Lambda$.
This implies that $a_i^m\l=\l$ for some positive integer $m$, and thus $a_i^m=1$. Let
$m_i$ be the least such number. Now $\tau_i\mapsto a_i$ defines a faithful action of
$C_{m_i}=\<\tau_i\>$,  the cyclic group of order $m_i$, on $\Lambda$.

By symmetry, multiplication with $\frac{1}{a_i}$ permutes $M$. As
$(a_i\l_j)(\frac{1}{a_i}\mu_j)=\l\mu$, the set $\{\l_j\}_{j=1}^l$ is invariant under the
$C_{m_i}$-action. Since every orbit has length $m_i$,  $l$ is a multiple of $m_i$. It
follows that $a_i$ is a $l$th root of unity. 
We conclude, that $a_i,\;i=1,\ldots,l$ are distinct $l$th roots of unity.

Considering a primitive $l$th root of unity $a_i$, we have $m_i=l$. Since $C_{m_i}$ acts
faithfully on $\Lambda$, $l$ divides $|\Lambda|=\deg f$. 
By an analogous argument, $l\mid\deg g$.
Moreover, from the identity $f(x)\star g(x)=m_k(\l\mu)(x)^l$ follows that
$l=\frac{\deg f \deg g}{\deg m_k(\l\mu)}$. 

If $\cha k=p>0$, then $k$ has only one $p$th root of unity. Thus for $l\in p\mathbb{Z}$, there
cannot exist $l$ different $l$th roots of unity in $k$. 
\end{proof}

As for the general case, $f(x)$ and $g(x)$ may be decomposed into irreducible factors over
$F$: $f(x)=\prod_i f_i(x)$ and $g(x)=\prod_j g_j(x)$. Taking $\l_i\in B$ and
$\mu_j\in C$ to be zeros of $f_i(x)$ and $g_j(x)$, by Proposition~\ref{linob} we get
\begin{equation} \label{allm}
  f(x)\star g(x)=\prod_{i,j}f_i(x)\star g_j(x)=\prod_{i,j}m_F(\l_i\mu_j)(x)^{l_{ij}}
\end{equation}
with $l_{ij}=\frac{\deg f_i \deg g_j}{\deg m_F(\l_i\mu_j)}$ for all $i,j$.
Note, however, that the factors $m_F(\l_i\mu_j)(x)$ are in general not in $k[x]$, even
though their product is.

The following result tells how to combine the factors $m_F(\l_i\mu_j)(x)$ in
Equation~(\ref{allm}) to obtain the irreducible factors of $f(x)\star g(x)$
over $k$.
The Galois group $\G{F}{k}$ acts on $F[x]$ by transforming the coefficients.
If $\sigma\in\G{F}{k}$ is an automorphism, the the image of $p(x)\in F[x]$ under $\sigma$
is denoted by $\sigma^\ast p(x)$.
\begin{pro}
  Let $\l\in \Lambda$ and $\mu\in M$. If
  $X=\G{F}{k}\cdot m_F(\l\mu)$ denotes the orbit of $m_F(\l\mu)$ under
  $\G{F}{k}$, then
  \begin{equation*}
    m_k(\l\mu)(x)=\prod_{h\in X}h(x).
  \end{equation*}
\end{pro}

\begin{proof}
Since $B$ and $C$ both are Galois extensions of $k$, so is $F=B\cap C$.
Hence $\sigma(F)=F$ for all $\sigma\in\G{E}{k}$, and the restriction map
$\G{E}{k}\to\G{F}{k},\:\sigma\mapsto\sigma|_F$ is an epimorphism of groups.

By Equation~(\ref{rotarunt}), we have
\begin{equation*}
  m_k(\l\mu)(x)=\prod_{\a\in\G{E}{k}\cdot\l\mu}(x-\a)=\prod_{h\in Y}h(x),
\end{equation*}
where
$Y=\{m_F(\a) \mid \a\in\G{E}{k}\cdot\l\mu\}= 
\{m_F(\sigma(\l\mu)) \mid \sigma\in\G{E}{k}\}$.
But, clearly, $m_F(\sigma(\l\mu))=\sigma|_F^\ast m_F(\l\mu)$, whence
\begin{equation*}
Y=\{\sigma|_F^\ast m_F(\l\mu) \mid \sigma\in\G{E}{k}\}=
\{\tau^\ast m_F(\l\mu) \mid \tau\in\G{F}{k}\}=X.
\end{equation*}
\end{proof}

As an illustration of the above, we consider the case when $\deg f=\deg g =2$.
Here, $[B:k]=[C:k]=2$, and $\G{B}{k}=\<\sigma\>$ and $\G{C}{k}=\<\tau\>$ are cyclic of
order two. We have $f(x)=(x-\l)(x-\sigma(\l))\in B[x]$ and
$g(x)=(x-\mu)(x-\tau(\mu))\in C[x]$.

First assume that $F=B\cap C=k$. By Proposition~\ref{linob}, 
$f(x)\star g(x)=m_k(\l\mu)(x)^l$, with $l\deg m_k(\l\mu)=\deg f \deg g = 4$ and $l\mid\deg
f$. 
Consequently, $(l,\deg m_k(\l\mu))\in\{(2,2),(1,4)\}$.
The case $l=2$ occurs when the zeros of 
\begin{equation*}
  f(x)\star g(x) = \left(x-\l\mu\right)\left(x-\l\tau(\mu)\right)\left(x-\sigma(\l)\mu\right)\left(x-\sigma(\l)\tau(\mu)\right)
\end{equation*}
have multiplicity two, whereas $l=1$ means that $f(x)\star g(x)$ is irreducible,
and thus has distinct zeros. Since $\sigma(\l)\mu\neq\l\mu\neq\l\tau(\mu)$, it follows
that $l=2$ if and only if $\l\mu=\sigma(\l)\tau(\mu)$.
If so, then $\frac{\l}{\sigma(\l)}=\frac{\tau(\mu)}{\mu}=a\in B\cap C=k$, and thus
$\l=a\sigma(\l)$, and $\l^2=a\l\sigma(\l)\in k$. Similarly, $\mu^2\in k$. This implies that
$f(x)=x^2-b$ and $g(x)=x^2-c$ for some $b,c\in k\setminus\{0\}$. Conversely, it is clear
that if $f(x)$ and $g(x)$ have this form, then $\sigma(\l)=-\l$ and $\tau(\mu)=-\mu$, hence
$\sigma(\l)\tau(\mu)=\l\mu$ and $l=2$. 

Assume instead $F=B=C$, and let $\kappa$ be the non-trivial automorphism of $F$. Then
\begin{equation*}
  f(x)\star g(x) =
  \left(x-\l\mu\right)\left(x-\kappa(\l\mu)\right)\left(x-\l\kappa(\mu)\right)\left(x-\kappa(\l)\mu\right)
  = h_1(x)h_2(x)
\end{equation*}
where $h_1(x)=(x-\l\mu)(x-\kappa(\l\mu))$ and $h_2(x)=(x-\l\kappa(\mu))(x-\kappa(\l)\mu)$
are distinct polynomials in $k[x]$.
Certainly, $h_1(x)$ is reducible if $\l\mu\in k$ and $h_2(x)$ is reducible if
$\l\kappa(\mu)\in k$. 
If both $h_1(x)$ and $h_2(x)$ are reducible, then
$(\l+\kappa(\l))\mu=\l\mu+\kappa(\l)\mu\in k$ and thereby, since $\l+\kappa(\l)\in k$ and
$\mu\not\in k$, $\kappa(\l)=-\l$. Similarly, $\kappa(\mu)=-\mu$ in this case. Hence
$f(x)\star g(x)$ splits into linear factors if and only if $f(x)=x^2-b$ and $g(x)=x^2-c$
for some (non-squares) $b,c\in k$.

As an example, let $k=\R$ and
$f(x)=g(x)=x^2-x+1\in\R[x]$. Then $F=B=C=\C$ and 
\begin{align*}
h_1(x)&=\left(x-\left(e^{i\frac{\pi}{3}}\right)^2\right)\left(x-\left(e^{-i\frac{\pi}{3}}\right)^2\right)=
x^2-x+1, \\
h_2(x)&=\left(x-e^{i\frac{\pi}{3}}e^{-i\frac{\pi}{3}}\right)^2=(x-1)^2, 
\quad\mbox{and thus} \\
f(x)\star g(x) &= (x^2-x+1)(x-1)^2.
\end{align*}
We find it notable that the irreducible factors in this product occur with different
multiplicities.

\subsection{Clebsch-Gordan formulae over real closed fields}

We conclude this section by stating the Clebsch-Gordan formulae for endomorphisms over
real closed fields. Apart from the existing solution for algebraically closed fields of
characteristic zero, this is the only situation in which we have been able to obtain
completely explicit formulae for the structure of the representation ring. 
Partly, this is due to the fact that over fields which are not algebraically nor real
closed, no convenient description of the irreducible polynomials is known, and hence, not
even a basis for the ring $\bar{R}$ can be explicitly written down.
As for algebraically closed fields of positive characteristic, the behaviour of $R$ is
completely determined by the subring $R'$. Here one could in principle hope to find a
closed formula for the product $v_iv_j$, improving upon the recursive descriptions given
in this article and in \cite{iwamatsu07a}.

A field is, by definition, \emph{real closed} if it has index two in its algebraic
closure. Any real closed field has characteristic zero, and the algebraic closure is
obtained by adjoining the square root of minus one. Assume, for what remains of this
section, that $k$ is real closed, and that $i\in K$ is a square root of minus one.
Let $\l\mapsto\bar{\l}$ denote the non-trivial automorphism of $K$ over $k$.

Any indecomposable endomorphism of a finite-dimensional vector space over $k$ can be
written, in a suitable basis, either as a Jordan block $J_\l(l)$ with $\l\in k$, or as
\begin{equation}
  \tilde{R}_\l(l) = 
  \left(\begin{smallmatrix}
    R_\l & \mathbb{I}_2 \\
    & \ddots & \ddots \\
    && R_\l & \mathbb{I}_2 \\
    &&& R_\l \\
  \end{smallmatrix} \right)
  \in k^{2l\times 2l}
\end{equation}
where $\l=a+bi\in K\setminus k$ and 
$R_\l=\left(\begin{smallmatrix} a&-b\\b&a\end{smallmatrix}\right)$.
By Lemma~\ref{separate}, $\tilde{R}_\l(l)\sim J_1(l)\otimes R_\l$. 
Consequently, $\tilde{R}_\l(l)\otimes \tilde{R}_\mu(m)\sim (J_1(l)\otimes J_1(m))\otimes
(R_\l\otimes R_\mu)$ for any $\l,\mu\in K\setminus k$ and $l,m\in\mathbb{N}\setminus\{0\}$.
The product $J_1(l)\otimes J_1(m)$ is completely governed by Theorem~\ref{jordan}. Thus,
to find the decomposition of $\tilde{R}_\l(l)\otimes \tilde{R}_\mu(m)$, is suffices to
determine the decomposition of $R_\l\otimes R_\mu$. As we shall see, this problem was
essentially solved in the previous section.

\begin{thm} \label{realjordan}
  Assume $k$ is real closed. Let $l$ and $m$ be positive integers, and $\l,\mu\in K$
  non-zero elements. The following formulae hold:
  \begin{enumerate}
  \item $J_\l(1)\otimes J_\mu(1) \sim J_{\l\mu}(1)$ if $\l,\mu\in k$,
  \item $J_\l(1)\otimes R_\mu \sim R_{\l\mu}$ if $\l\in k$, $\mu\in K\setminus k$.
  \end{enumerate}
Suppose $\l,\mu\not\in k$. Then
\begin{enumerate}
  \setcounter{enumi}{2}
  \item $R_\l\otimes R_\mu\sim 2 J_{\l\mu}(1) \oplus 2 J_{-\l\mu}(1)$ 
    if $\l,\mu\in \spn_k \{i\}$,
  \item $R_\l\otimes R_\mu \sim 2 J_{\l\mu}(1) \oplus R_{\bar{\l}\mu}$ 
    if $\l,\mu\not\in\spn_k\{i\}$, $\bar{\l}=r\mu$ for some $r\in k$, 
  \item $R_\l\otimes R_\mu \sim 2 J_{\bar{\l}\mu}(1) \oplus R_{\l\mu}$ 
    if $\l,\mu\not\in\spn_k\{i\}$, $\l=r\mu$ for some $r\in k$, 
  \item $R_\l\otimes R_\mu \sim R_{\l\mu}\oplus R_{\bar{\l}\mu}$ otherwise.
  \end{enumerate}
\end{thm}

\begin{proof}
The two first relations are immediate, since $J_\l(1)$ is simply a scalar.
For the remaining three, set $f(x)=m_k(\l)(x)=(x-\l)(x-\bar{\l})$ and
$g(x)=m_k(\mu)(x)=(x-\mu)(x-\bar{\mu})$, where $\l,\mu\in K\setminus k$. The problem is to determine the decomposition of
$f(x)\star g(x)$ into irreducible factors over $k$. Clearly, $K$ is a common splitting
field of $f(x)$ and $g(x)$. Hence, we are in the situtation treated the second to last
paragraph of Section~\ref{Rstreck}.

The polynomial $f(x)\star g(x)=h_1(x)h_2(x)$ decomposes into linear factors if and only
if both $f(x)$ and $g(x)$ have trace zero, \ie, if $\l,\mu\in \spn_k\{i\}$. In this case,
$f(x)\star g(x)=(x-\l\mu)^2(x-\bar{\l}\mu)^2=(x-\l\mu)^2(x-(-\l\mu))^2$, which corresponds
to the third relation above.

The condition $\l\mu\in k$ is equivalent to $\bar{\l}\in\spn_k\{\mu\}$. Hence, $h_1(x)$ is
reducible, but $h_2(x)$ irreducible, precisely when $\bar{\l}\in\spn_k\{\mu\}$ and
$\l,\mu\not\in\spn_k\{i\}$. This gives the fourth clause in the theorem. 
Similarly, clause five corresponds to the the case when $h_1(x)$ is irreducible and
$h_2(x)$ reducible. 
In the remaining situation, $h_1(x)$ and $h_2(x)$ are both irreducible, giving 
$f(x)\star g(x)=m_k(\l\mu)(x)\:m_k(\bar{\l}\mu)(x)$.
\end{proof}

\section{Representation ring of $\tilde{\mathbb{A}}_n$}
In this section we consider the generalisation of the representation ring of $k[x]$ to the representation rings of quivers.

A quiver $Q$ is an oriented graph, \ie\ it consists of a set of vertices $Q_0$, a set of arrows $Q_1$ and two maps $t,h : Q_1\rightarrow Q_0$, mapping each arrow $\alpha$ to its tail $x= t \alpha$ and head $y =h \alpha$ respectively. We depict this by $x \overset{\alpha}{\rightarrow} y$.

Let $Q$ be a quiver. A representation $V$ of $Q$ over $k$ consists of a  finite-dimensional vector space $V_x$ over $k$, for each $x \in Q_0$ and a $k$-linear map $V(\alpha): V_x \rightarrow V_y$, for each arrow $x \overset{\alpha}{\rightarrow} y$ in $Q$. The direct sum and tensor product of quiver representations are defined pointwise and arrowwise, \ie\ for representations $V$ and $W$ of $Q$,
\[
(V \oplus W)_x = V_x \oplus W_x\;, \qquad
(V \oplus W)(\alpha)= V(\alpha) \oplus W(\alpha)
\]
and
\[
(V \otimes W)_x = V_x \otimes W_x\;, \qquad
(V \otimes W)(\alpha)= V(\alpha) \otimes W(\alpha).
\]
A quiver representation is indecomposable if it is non-zero and only decomposes trivially into a direct sum of two subrepresentations. Since the tensor product is distributive over direct sums, we can define the representation ring $R(Q)$ similarly as for $k[x]$. See \cite{herschend07b} for details.

Let $Q$ be the loop quiver:
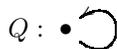
\[
Q: \xymatrix{\bullet \ar@(dr,ur)}
\]
The representations of $Q$ correspond naturally to linear operators and thus to $k[x]$-modules. In fact this correspondence extends to an additive equivalence of categories that respects the tensor product. Hence $R(Q) \iso R$.

We will now extend our results on $R$ to the representation rings of quivers of extended Dynkin type $\tilde{\mathbb{A}}$. Let $n \in \mathbb{N}$ and $Q$ be a quiver of type $\tilde{\mathbb{A}}_n$, \ie\ a quiver whose underlying graph is
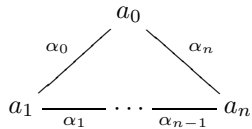
\[
\xymatrix{& a_0 \edge{dl}_{\alpha_0} &\\ a_1 \edge{r}_{\alpha_1} & \cdots \edge{r}_{\alpha_{n-1}}& a_n \edge{ul}_{\alpha_{n}}}
\]

Recall the well-known classification of indecomposable representations of $Q$, which
can be found in \cite{gabriel92}. For integers $i$ and $j$ such that $i \le j$ we define
the representation $S(i,j)$ of $Q$ as follows. For each $x \in Q_0$ the vector space
$S(i,j)_x$ has the basis $\{e_s \;|\; i \le s \le j, \; s \equiv x \mod n+1\}$. Thus
$\bigoplus_{x \in Q_0} S(i,j)_x$ has the basis $\{e_s \;|\; i \le s \le j\}$. The linear
map $S(i,j)(\alpha)$ maps a basis vector $e_s \in S(i,j)_{t\alpha}$ to $e_{s+1}$ or
$e_{s-1}$ depending on the orientation of $\alpha$. The representations $S(i,j)$ are all
indecomposable and called strings. Two strings $S(i,j)$ and $S(i',j')$ are isomorphic if
and only if $i \equiv i' \mod n+1$ and $j-i = j'-i'$. 

For each positive integer $s$ and irreducible monic polynomial $f(x) \in k[x]$, with
$f(0) \not = 0$ define the representation $B_{f}(s)$ of $Q$ by 
\[
\begin{split}
B_{f}(s)_{a_i} &= k[x]/f(x)^s \\
B_{f}(s)(\alpha_i) &= \begin{cases}
  1 & \mbox{if }  i\not = n, \\ 
x & \mbox{if }  i=n. \end{cases}
\end{split}
\]
The representations $B_{f}(s)$ are all indecomposable and pairwise non-isomorphic. They are called bands. 

\begin{thm}\label{tA_nclass}\cite[p.121]{gabriel92}
The set of all strings and bands classifies indecomposable representations of $Q$, up to isomorphism.
\end{thm}

For all $i,j \in \mathbb{Z}$ set $i\wedge j = \min\{i,j\}$. Further, for each $q \in \mathbb{Q}$ denote the integer part of $q$ by $[q]$. The following result solves the Clebsch-Gordan problem for $Q$.

\begin{thm}\label{cgformA}\cite{herschend07}
For all integers $i,i',j,j'$ such that $0 \leq i\leq i' \leq n$, $i\leq j$ and $i' \leq j'$; irreducible monic polynomials $f(x), g(x) \in k[x]$ with non-zero constant terms and positive integers $s,t$, the following formulae hold.
\renewcommand{\labelenumi}{\emph{(\roman{enumi})}}
\begin{enumerate}
\item
\[
\begin{split}
S(i,j) \otimes S(i',j') \iso &\bigoplus_{k=0}^{\left[ \frac{j'-i}{n+1}\right]} S(i, j\wedge (j'-k(n+1))) \\ \oplus &\bigoplus_{k = 1}^{\left[ \frac{j-i'}{n+1}\right]}S(i', j' \wedge (j-k(n+1)))
\end{split}
\]
\item
\[
S(i,j)\otimes B_{f}(s) \iso s (\deg f)S(i,j) 
\]
\item
\[
B_{f}(s) \otimes B_{g}(t) \iso \bigoplus_{j\in J} B_{h_j}(d_j),
\]
where
\[
k[x]/f(x)^s \otimes k[x]/g(x)^t \iso \bigoplus_{j\in J} k[x]/h_j(x)^{d_j}
\]
is the decomposition of $k[x]/f(x)^s \otimes k[x]/g(x)^t$ into indecomposable $k[x]$-modules.
\end{enumerate}
\renewcommand{\labelenumi}{(\roman{enumi})}
\end{thm}
It should be noted that the Theorem \ref{cgformA} is only stated for $k$ algebraically closed fields of characteristic zero in \cite{herschend07}. However, the generalisation stated here is not essential, as can be concluded from the comments in \cite{herschend07}.

From Theorem \ref{cgformA} follows that the strings $S(i,j)$ span an ideal $I_n$ in $R(Q)$
on which the band $B_{f}(l)$ acts by multiplication by $l \deg f$. Moreover the bands span
a subring in $R$, which is isomorphic to $R' \otimes_\mathbb{Z} \bar{R}$. Thus we obtain the
following result.

\begin{thm}\label{repringA}
Let $Q$ be a quiver of type $\tilde{\mathbb{A}}_n$, Then
\[
R(Q) \iso R' \otimes_\mathbb{Z} \bar{R} \oplus I_n,
\]
where the ring structure of $R' \otimes_\mathbb{Z} \bar{R} \oplus I_n$ is defined by $(a\otimes b) w = \dim(a)\dim(b)w$ for all $a \in R'$, $b \in \bar{R}$ and $w \in I_n$.
\end{thm}
Depending on the ground field this description can of course be made more precise using the results form previous sections.

The close relationship between $R$ and $R(Q)$ is connected to the fact that indecomposable
$k[x]$-modules appear in the construction of indecomposable representations of $Q$. This
fact is not specific to quivers of type $\tilde{\mathbb{A}}$ but inherent in all tame
quivers with relations. Thus it seems probable that a connection, similar to the one given
in this section, can be found in other tame cases as well. Indeed this is the case for the
double loop quiver 
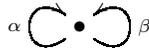
\[
\xymatrix{\bullet \ar@(dl,ul)^{\alpha}\ar@(dr,ur)_{\beta}}
\]
with relations $\beta\alpha = \alpha\beta = \alpha^n = \beta^n = 0$, in the sense that the Clebsch-Gordan problem for this quiver with relations contains the Clebsch-Gordan problem for $k[x]$ as a subproblem \cite{herschend07}.  

\vspace*{0.5cm}
\noindent
{\bf Acknowledgements} The second author is grateful for the financial support provided by
the Japan Society for the Promotion of Science, which enabled the stay at Nagoya
University during which this article partly was written. Both authors are indebted to
Karin Erdmann for pointing out the very valuable reference \cite{green61} and grateful to
Yuji Yoshino who noted the relevance of the absolute Galois group. 

\bibliographystyle{plain}
\bibliography{biban}

\def\cprime{$'$}
\begin{thebibliography}{10}

\bibitem{aitken35}
A.~C. Aitken.
\newblock The normal form of compound and induced matrices.
\newblock {\em Proc. London Math. Soc.}, 38:354--376, 1935.

\bibitem{deuring33}
M.~Deuring.
\newblock Galoissche {T}heorie und {D}arstellungstheorie.
\newblock {\em Math. Ann.}, 107(1):140--144, 1933.

\bibitem{gabriel92}
P.~Gabriel and A.~V. Roiter.
\newblock Representations of finite-dimensional algebras.
\newblock In {\em Algebra, VIII}, volume~73 of {\em Encyclopaedia of
  Mathematical Sciences}, pages 1--177. Springer, 1992.
\newblock With a chapter by B. Keller.

\bibitem{green61}
J.~A. Green.
\newblock The modular representation algebra of a finite group.
\newblock {\em Illinois J. Math}, 6:607--619, 1962.

\bibitem{herschend07b}
M.~Herschend.
\newblock On the representation ring of a quiver.
\newblock {\em Algebr. Represent. Theory}.
\newblock to appear.

\bibitem{herschend07}
M.~Herschend.
\newblock Galois coverings and the {C}lebsch-{G}ordan problem for quiver
  representations.
\newblock {\em Colloq. Math.}, 109(2):193--215, 2007.

\bibitem{huppert90}
B.~Huppert.
\newblock {\em Angewandte lineare Algebra}.
\newblock Walter de Gruyter \& Co., 1990.

\bibitem{iwamatsu07a}
K.~Iima and R.~Iwamatsu.
\newblock On the {J}ordan decomposition of tensored matrices of {J}ordan
  canonical forms, 2006.
\newblock http://arxiv.org/abs/math/0612437.

\bibitem{MaVl}
A.~Martsinkovsky and A.~Vlassov.
\newblock The representation rings of k[x].
\newblock preprint.

\bibitem{rotman}
Joseph Rotman.
\newblock {\em Galois theory}.
\newblock Universitext. Springer-Verlag, New York, second edition, 1998.

\bibitem{wakamatsu03}
Takayoshi Wakamatsu.
\newblock On graded {F}robenius algebras.
\newblock {\em J. Algebra}, 267(2):377--395, 2003.

\end{thebibliography}

\vspace*{1cm}
\noindent
\begin{tabular}[t]{@{}l}{\it Erik Darp\"o}\vspace*{0.2cm} \\
{\it Department of Mathematics}\\{\it Uppsala University, Box 480}\\
{\it SE-751 06 Uppsala}\\{\it Sweden}\vspace*{0.2cm}\\
{\tt erik.darpo@math.uu.se}
\end{tabular}
\hfill
\begin{tabular}[t]{l@{}}{\it Martin Herschend}\vspace*{0.2cm}\\
{\it Graduate School of Mathematics}\\
{\it Nagoya University, Chikusa-ku}\\
{\it Nagoya, 464-8602}\\
{\it Japan}\vspace*{0.2cm} \\
{\tt martin.herschend@math.uu.se}
\end{tabular}

\end{document}